\documentclass[12pt]{amsart}
\usepackage{amsmath}
\usepackage{amssymb}
\usepackage{stmaryrd}
\usepackage{epsfig,color,esint}
\usepackage{blindtext}
\usepackage{enumerate}
\usepackage{enumitem}
\usepackage{hyperref}
\usepackage{graphicx}
\usepackage{url}
\usepackage{bbm}
\usepackage{nicefrac,mathtools}
\usepackage{mathrsfs}

\usepackage{comment}
\usepackage{pgfplots}
\usepackage{tikz}
\usetikzlibrary{arrows,shapes,trees,backgrounds}

\theoremstyle{definition}

\newtheorem{theorem}{Theorem}[section]

\newtheorem{example}[theorem]{Example}

\newtheorem{prop}{Proposition}[section]
\newtheorem{theo}{Theorem}[section]
\newtheorem{lemm}{Lemma}[section]
\newtheorem{coro}{Corollary}[section]
\newtheorem{rema}{Remark}[section]

\newtheorem{defi}{Definition}[section]

\newtheorem{pro}{Problem}[section]

\newcommand{\beq}{\begin{equation}}
\newcommand{\eeq}{\end{equation}}
\newcommand{\beqs}{\begin{eqnarray*}}
\newcommand{\eeqs}{\end{eqnarray*}}
\newcommand{\beqn}{\begin{eqnarray}}
\newcommand{\eeqn}{\end{eqnarray}}
\newcommand{\beqa}{\begin{array}}
\newcommand{\eeqa}{\end{array}}

\newtheorem*{remark*}{Remark}

\numberwithin{equation}{section}

\def\begeq{\begin{equation}}
\def\endeq{\end{equation}}

\def\S{\mathbb S}

\def\p{\partial}

\def\Om{\Omega}

\def\D{\nabla}

\def\S{\mathbb S}

\def\F{\mathcal F}

\def\Om{\Omega}

\def\D{\nabla}

\makeatletter
\@namedef{subjclassname@2020}{\textup{2020} Mathematics Subject Classification}
\makeatother

\title{Limiting shape of the $L_p$-Minkowski problem}
\author{Shi-Zhong Du}
\address{The Department of Mathematics,
            Shantou University, Shantou, 515063, P. R. China.}
\email{szdu@stu.edu.cn}

\author{Xu-Jia Wang}
\address{Institute for Theoretical Sciences, Westlake University,
Hangzhou 310024,  China.}
\email{wangxujia@westlake.edu.cn}

\author{Bao-Cheng Zhu}
\address{School of Mathematics and Statistics, Shaanxi Normal University, Xi'an, 710119, China}
\email{bczhu@snnu.edu.cn}

\thanks{The first author was  supported by NSFC (12171299);
the second author was supported by the start-up fund from Westlake University;
the third author was supported by NSFC (12371060) and the Sydney Mathematical Research Institute at USYD}


\begin{document}

\begin{abstract}
In \cite{A03}, Andrews classified the limiting shape for isotropic curvature flow corresponding to the solutions of the $L_p$-Minkowski problem as $p\to-\infty$ in the planar case. In this paper, we use the group-invariant method to study the asymptotic shape of solutions to the $L_p$-Minkowski problem as $p\to-\infty$ in high dimensions. For any regular polytope $T$, we establish the existence of a solution ${\Omega^{(p)}}$ to the $L_p$-Minkowski problem that converges to $T$ as $p\to-\infty$, thereby revealing the intricate geometric structure underlying this limiting behavior.
We also extend the result to the dual Minkowski problem.
\end{abstract}

\maketitle\markboth{Limiting shape of the $L_p$-Minkowski problem}
{Limiting shape of the $L_p$-Minkowski problem}



\baselineskip16.8pt
\parskip5pt

\section{Introduction}

The classical Minkowski problem looks for a convex body such that its surface measure matches a given Radon measure $\mu$ on the unit sphere ${\mathbb{S}}^{n}$ in $\mathbb{R}^{n+1}$.
This problem dates back to the early works of Minkowski, Aleksandrov and Fenchel-Jessen (see \cite{Sch14}).
For a convex body $\Omega\in{\mathcal{K}}_0$  (the set of convex bodies containing the origin) and two parameters $p,q\in{\mathbb{R}}$,
Lutwak,Yang and Zhang \cite{LYZ18} introduced the measure
$$
\mathcal{C}_{p,q}(\Omega,\omega)=\int_{G^{-1}(\omega)}\frac{|y|^{q-n-1}}{h^{p-1}(G(y))}
d{\mathcal{H}}^{n},$$
for any measurable set $\omega\subset {\mathbb{S}}^{n}$,
and proposed  the following dual Minkowski problem.
Given a finite Radon measure $\mu$ on ${\mathbb{S}}^n$,
find a convex body $\Omega\in{\mathcal{K}}_0$ such that $d\mathcal{C}_{p,q}(\Omega,\cdot)=d\mu$.
 Here, $G: \partial \Omega \rightarrow {\mathbb{S}}^n$ is the Gauss map, $h: {\mathbb{S}}^{n}\to{\mathbb{R}}$ is the support function of $\Om\in \mathcal K_0$, and $\mathcal{H}^n$ is the $n$-dimensional Hausdorff measure.
In the smooth category, the measure $\mu$ is given by $d\mu=fd\sigma$ for a   positive function $f$, where $d\sigma$ is the spherical Lebesgue measure,
the $L_p$ dual Minkowski problem can then be formulated by the equation
  \begin{equation}\label{e1.1}
    \det(\nabla^2h+hI)=fh^{p-1}(h^2+|\nabla h|^2)^{\frac{n+1-q}{2}}, \ \ \forall\ x\in{\mathbb{S}}^{n}.
  \end{equation}
This is a fully nonlinear equation of Monge-Amp\`{e}re type,
which has received increasing attention in recent years (see e.g. \cite{BF19,CHZ,HLYZ,HZ2018,LYZ18,Zhao17,Zhao18}).
When $q=n+1$, equation \eqref{e1.1} reduces to the $L_p$-Minkowski problem, introduced in \cite{Lut93}, given by
    \begin{equation}\label{e1.2}
     \det(\nabla^2h+hI)=fh^{p-1}, \ \ \forall\ x\in{\mathbb{S}}^{n}.
    \end{equation}
Here, $\nabla$ denotes the standard Levi-Civita connection on the sphere ${\mathbb{S}}^{n}$, and $I$ is the identity matrix.

The $L_p$-Minkowski problem has been extensively studied in recent years (see e.g. \cite{A99,A00,BBCY,BLYZ13,BIS19,Du21,HLW16,HLX15,HLYZ05,IM24,LW13,LYZ04,M24a,M24b,Zhu15,Zhu17}).
According to the Blaschke-Santal\'{o} inequality,
the problem is divided into three cases, the sub-critical case $p>-n-1$,
the critical case $p=-n-1$, and the super-critical case $p<-n-1$.
There are rich phenomena on the existence and multiplicity of solutions.
In the sub-critical case $p\in(-n-1,0]$, the solution is usually not unique \cite{JLW15}.
When $p<-n$, there may be infinitely many solutions \cite {Li19}.
In the critical case $p=-n-1$, due to a Kazdan-Warner type obstruction \cite{CW06},
there may be no solutions to the problem \eqref{e1.2}.
In the super-critical case $p<-n-1$,
surprisingly, Guang, Li and the second author \cite{GLW2024} proved that
for any $p<-n-1$ and any positive, smooth function $f$, there is a solution to \eqref{e1.2}.
An amazing result was obtained by Andrews \cite{A03}.
He proved that in the planar case $n=1$,
for any regular polygon $T_k$ (with $k$-sides),
there is a solution to \eqref{e1.2} which converges to $T_k$, as $p$ tends to negative infinity. This raises a natural question of whether such limiting behavior holds in high dimensions:

\begin{pro}\label{problem}
(i) Does there exist a solution ${h^{(p)}}$ to the $L_p$-Minkowski problem \eqref{e1.2}, such that the associated convex body ${\Omega^{(p)}}$ converges to a polytope $T$, which is tangential to ${\mathbb{S}}^{n}$, as $p\to -\infty$?\\
(ii) Moreover, for any regular polytope $T$ tangent to the unit sphere $ {\mathbb{S}}^{n}$,
does there exist a solution ${h^{(p)}}$ to \eqref{e1.2}   for negative large $p$,
such that the associated convex body ${\Omega^{(p)}}$ converges to the polytope $T$
as $p\to-\infty$?
\end{pro}

In fact, Andrews classified the limiting shape of the solutions to \eqref{e1.2} in the planar case by solving an isotropic curvature flow.
In this paper, we explore a new method, called the group-invariant method, to study the limiting shape of the solutions to \eqref{e1.2} in high dimensions. To do so, we assume that,
throughout the paper, $f$ satisfies the following conditions:
\begin{itemize}
\item [(i)] $f$ is a measurable function satisfying
\beq\label{fcc}
c_1 \le  f \le c_2
\eeq
for positive constants $c_1, c_2$.

\item [(ii)] $f$ is invariant with respect to a group ${\mathcal{S}}$,
which is a discrete subgroup of $O(n+1)$ satisfying the spanning property  (introduced in Section 2).
\end{itemize}
A particular case is when $\mathcal S$ is the discrete subgroup of $O(n+1)$
generated by a regular polytope.

First, we obtain the asymptotic behavior of solutions to the $L_p$-Minkowski problem \eqref{e1.2}
for negative large $p$ in high dimensions, which solves part (i) of Problem \ref{problem}.

\begin{theo}\label{t1.1}
Let ${\mathcal{S}}$ be a discrete subgroup of $O(n+1)$ satisfies the spanning property.
Then for $p<-n-1$, there exists an ${\mathcal{S}}$-invariant solution ${h^{(p)}}$ to the $L_p$-Minkowski problem \eqref{e1.2}, such that the associated convex body ${\Omega^{(p)}}$ sub-converges to a group invariant polytope $T$, which is tangential to ${\mathbb{S}}^{n}$, as $p\to -\infty$.
\end{theo}

We say that a polytope $T$ is tangential to the unit sphere ${\mathbb{S}}^{n}$
if every face of $T$ is tangential to ${\mathbb{S}}^{n}$.
It is easy to see that if $T$ is tangential to ${\mathbb{S}}^{n}$,
then $T$ is convex and $T$ contains the unit ball  ${B_1}$.

By the regularity theory of Caffarelli \cite{caffarelli1990localization,caffarelli1990interior},
the solution $h^{(p)}$ is $C^{1,\alpha}$ if $f$ is positive and bounded,
and is $C^{2,\alpha}$ if $f$ is  also H\"{o}lder continuous.
By the maximum principle, there is a unique solution $h^{(p)}$ to \eqref{e1.2} when $p>n+1$ \cite{CW06}.
It is easy to see that the solution $h^{(p)}$ tends to constant one when $p$ tends to infinity.
In fact,  at the maximum point of $h^{(p)}$, we have
$$ h^{p-n-1}\le \det (\D^2 h+hI)/(fh^n)\le 1/f,\ \ h=h^{(p)}$$
by equation \eqref{e1.2}. Hence $\lim_{p\to+\infty} \sup h^{(p)}=1$.
Similarly, one can show that $\lim_{p\to+\infty} \inf h^{(p)}=1$ by evaluating \eqref{e1.2} at the minimum point of $h^{(p)}$.
We conclude that the limiting shape of the solutions $\Omega^{(p)}$ is the unit ball  ${B_1}$ as $p\to+\infty$.

Theorem \ref{t1.1} deals with the case $p<-n-1$.
When $f\equiv1$, there is  a trivial solution $h\equiv 1$ to \eqref{e1.2}.
Theorem \ref{t1.1} tells that there is a non-trivial solution $h^{(p)}$ when $p$ is negative large.
The solution $h^{(p)}$ is a maximizer of the functional
 \begin{equation}\label{f1.3}
 \F_{p}(\Omega) = V(\Omega)\Big(\int_{{\mathbb{S}}^{n}}fh^p{\Big/}\int_{{\mathbb{S}}^{n}}f\Big)^{-(n+1)/p}
 \end{equation}
 among all $\mathcal{S}$-invariant convex bodies $\Om$ containing the origin,
 where $h=h_\Om$ is the support function of $\Om$,
 $$V(\Om):= \int_{{\mathbb{S}}^n}h\det(\nabla^2h+h I). $$
 Here the value $V(\Om)$ is equal to $(n+1)$ times the volume $|\Om|$.
 The spanning property of the group $\mathcal S$
guarantees the uniform estimate of group invariant solutions,
and one can obtain the existence of solutions for all $p\ne n+1$.
By property of the functional $\F_{p}$, we show that the maximizer $h^{(p)}$ is not a constant when $p$ is negative large, and converges to a limit as $p$ tends to negative infinity. Moreover, the limit is a polytope.

The next theorem shows that every regular polytope tangent to the unit sphere
is a limit polytope, thereby resolving part (ii) of Problem \ref{problem}.

\begin{theo}\label{t1.2}
For any regular polytope $T$ tangential to the unit sphere $ {\mathbb{S}}^{n}$,
there exists solution ${h^{(p)}}$ to \eqref{e1.2}   for negative large $p$,
such that the associated convex body ${\Omega^{(p)}}$ converges to the polytope $T$
as $p\to-\infty$.
\end{theo}

A regular polytope has strong symmetry. We prove Theorem \ref{t1.2} by considering maximizers of the functional $\F_{p}$
among convex bodies with the strong symmetry.
Theorems \ref{t1.1} and \ref{t1.2}
can be regarded as a higher dimensional extension of
Andrews's result in (\cite{A03}, Theorem 1.5) in the planar case.
Here we allow $f$ to be a measurable function.

Next we consider the dual Minkowski problem \eqref{e1.1}.
Under the group invariant assumptions in Theorem \ref{t1.1},
we have uniform estimate for the group invariant solution $h^{(p, q)}$ to \eqref{e1.1},
which is a maximizer of the functional
\begin{equation}\label{f1.4}
 \F_{p,q}(\Omega) = V_q(\Omega)\bigg(\int_{{\mathbb{S}}^{n}}fh^p{\Big/}\int_{{\mathbb{S}}^{n}}f\bigg)^{-q/p}
 \end{equation}
 among all $\mathcal{S}$-invariant convex bodies $\Om\subset \mathcal K_0(\mathcal{S})$,
where
$$V_q(\Om)=\int_{\S^n}   r^q(y) dy , $$
{ and $V_{n+1}(\Om)=V(\Om)$.}

Here we are interested in the asymptotic behaviour of the solution $h^{(p, q)}$ as $p$ tends to negative infinity for fixed $q$, and the asymptotic behaviour as $q$ tends to positive infinity for fixed $p$.

\begin{theo}\label{t1.3}
Let $\mathcal{S}$ be as in Theorem \ref{t1.1}.
Let $h^{(p, q)}$ be the maximizer of the functional $\F_{p,q}$ and $\Om^{(p, q)}$ be the associated convex body.

\begin{itemize}

\item [(1)]
For any given $q>0$,
$\Om^{(p, q)}$  sub-converges to a group invariant polytope $T$,
which contains the unit ball ${B_1}$ and is tangential to $\S^n$,  as $p\to-\infty$.

\item [(2)]
For any given $q<0$,
$\Om^{(p, q)}$ converges to the unit ball as $p\to-\infty$.

\item [(3)]
For any given $p\neq 0$, $\Om^{(p, q)}$  sub-converges to a group invariant polytope $T$, which contains the unit ball ${B_1}$ and is tangential to $\S^n$, as $q\to+\infty$.

\end{itemize}
\end{theo}

 The next theorem shows that every regular polytope tangential to the unit sphere is a limit polytope.

\begin{theo}\label{t1.4}
Let $T\subset{\mathbb{R}}^{n+1}$ be a regular polytope and  let $(p,q)\in{\mathbb{R}}^2$ satisfy $pq\not=0$.
\begin{itemize}
\item [(1)]
If $T$ is tangential to the unit sphere ${\mathbb{S}}^n$,
there exists a local maximizer  $\Omega^{(p,q)}$ of the functional $V_{q}$,
such that  when $q>0$, $\Omega^{(p,q)}$ converges to $T$ as $p\to-\infty$.
\item [(2)]
When $p\ne 0$,  there exists a local maximizer $\Omega^{(p,q)}$ of   $V_q$,
which converges to a regular polytope similar to $T$ as $q\to+\infty$.
\end{itemize}
\end{theo}

This paper is organized as follows.
In Section 2, we introduce some notations and properties concerning subgroups $\mathcal{S}\subset O(n+1)$
and the corresponding group invariant polytopes.
In Section 3, we prove the existence of non-trivial group invariant solutions to \eqref{e1.2}
by a variational method for negative large $p$.
By passing to the limit $p\to-\infty$, we characterize the limiting convex body $\Omega^{(-\infty)}$ in Section 4.
In Section 5, we show that a regular polytope $T$
is a strictly local maximizer of the functional $\F_{p}$ and thus give a proof of Theorem \ref{t1.2}.
Finally, we consider the $L_p$ dual Minkowski problem \eqref{e1.1}
and prove Theorems \ref{t1.3}-\ref{t1.4} in Sections 6-7.

\vspace{10pt}

\section{Discrete subgroup  and group invariant polytope}

\subsection{Notations}
\

We will work in Euclidean space, where $|x|=\sqrt{x\cdot x}$ for any $x\in\mathbb{R}^{n+1}$. A compact convex set with non-empty interior in $\mathbb{R}^{n+1}$ is called a convex body.
Denote by ${\mathcal{K}}_0$ the set of convex bodies in ${\mathbb{R}}^{n+1}$ containing the origin.
Given a convex body $\Om\in \mathcal K_0$,
the support function $h: {\mathbb{S}}^{n}\to{\mathbb{R}}$ of $\Omega$ is defined by
$$h(x)=\max\big\{z\cdot x\ |\ z\in\Omega\big\},\ \ \forall x\in{\mathbb{S}}^{n}. $$
We denote by $h_{ij}=\nabla^2_{ij}h$ the Hessian matrix of $h$
with respect to an orthonormal frame $\{e_i\}_{i=1}^n$ on the unit sphere ${\mathbb{S}}^{n}$, where
$\delta_{ij}$ is the Kronecker delta symbol.
When $\Om$ is uniformly convex, the matrix $(A_{ij}):= (h_{ij}+h\delta_{ij})$ is positive definite,
with inverse $(A^{ij})$.
Let the radial distance function of $\partial \Omega$ be given by
$$
    r(y)=\sup\big\{\lambda>0\ |\ \lambda y\in\Omega\big\}, \ \forall \ y\in{\mathbb{R}}^{n+1}.
$$
There holds $r(y)y=h(x)x+\nabla h(x)$, where $x$ is the unit outer normal of $\p\Om$ at the point $ r(y)y\in\p\Om$,
such that $G: \  r(y)y\to x$ is the Gauss map.
Thus, $\Om$ can be recovered from $h$ by
$$\p\Om=\{h(x)x+\nabla h(x)\ |\ x\in{\mathbb{S}}^{n}\}. $$
The surface measure of $\partial\Omega$ is given by
$$
dS= \det(\nabla^2h+hI)d\sigma,
$$
where $d\sigma$ is the surface measure of the sphere ${\mathbb{S}}^{n}$.

\subsection{Discrete subgroup  and group invariant polytope}
\

In this subsection, we introduce some notations and properties about the group invariant polytopes
that will be used later.
According to \cite{Serge2002}, a group is a set that is closed under multiplication, and  satisfies
 the associative law, has a unit element, and has inverses.
By definition, each discrete subgroup ${\mathcal{S}}$ of $O(n+1)$ is a finite group.
For example,
 $$
  {\mathcal{S}}_a = \bigg\{\left(
    \begin{array}{cc}
      \cos\theta & \sin\theta\\
      -\sin\theta & \cos\theta
    \end{array}
    \right), \ \ \theta=j\times 2\pi/a, \ \ j\in{\mathbb{Z}}\bigg\}\subset O(2)
 $$
is a discrete subgroup when $a$ is rational, and an infinite subgroup when $a$ is irrational.

Let ${\mathcal{S}}$ be a discrete subgroup of $O(n+1)$.
We say that ${\mathcal{S}}$ satisfies the {\it spanning property}
if for any $a\in{\mathbb{S}}^{n}$, the set
$$P_a:= \text{conv}\{\phi(a), \phi\in{\mathcal{S}}\}$$
forms  a non-degenerate, $(n+1)$-dimensional polytope
in   ${\mathbb{R}}^{n+1}$.
Here we denote by $\text{conv} (\Psi)$ the convex hull of the set $\Psi$.
Moreover, we also define the bounded ratio condition by
    \begin{equation}\label{e2.1}
      \sup_{a\in{\mathbb{S}}^{n}}\gamma_a<\infty,
      \ \ \gamma_a:=\max_{{\mathbb{S}}^{n}}h_a/\min_{{\mathbb{S}}^{n}}h_a,
    \end{equation}
for the support function $h_a$ of $P_a$. We have the following equivalence.

\begin{prop}\label{p2.1}
 For a discrete subgroup ${\mathcal{S}}$ of $O(n+1)$,
the spanning property is equivalent to the bounded ratio condition \eqref{e2.1}.
\end{prop}

\begin{proof} If the spanning property does not hold for some $a\in{\mathbb{S}}^{n}$,
then $P_a$ is a degenerate polytope in $(n+1)$-dimension.
In this case,  the bounded ratio condition fails.

Conversely, if the bounded ratio condition does not hold for a sequence $a_j\in{\mathbb{S}}^{n}$
with $a_j$ tending to some $a\in{\mathbb{S}}^{n}$, it is clear that $P_a$ is degenerate.
\end{proof}

The spanning property is important in the sense that
it guarantees the  {\it a-priori} estimates for the maximizing sequence of our variational problem in Section 3.
Let us point out that in the spanning property,
the set $P_a= \text{conv}\{\phi(a), \phi\in{\mathcal{S}}\}$ is generated by only one arbitrary point $a$.

For a given discrete subgroup ${\mathcal{S}}\subset O(n+1)$ satisfying the spanning property,
there may be more than one polytope corresponding to the  subgroup ${\mathcal{S}}$.
For example,
let $\mathcal S\subset O(2)$ be the discrete subgroup associated with the regular $k$-polygon for $k=5$,
then all regular $k$-polygons with $k=5m$ are also $\mathcal S$-invariant, for all integers $m\ge 2$.

Regular polytopes are of particular interest.
According to \cite{coxeter1973regular,cromwell1997polyhedra},
we introduce the following definition for regular polytopes.

\begin{defi}
 A flag is a sequence of faces of a polytope $T$,
 each contained in the next, with exactly one face from each dimension.
 More precisely, a flag $\psi$ of an $n$-dimensional polytope $T$ is a set $\{F_0, \cdots, F_{n+1}\}$,
 where $F_i$ for $0\leq i\leq n+1$ is the $i$-dimensional face of $T$, such that $F_i\subset F_{i+1}$ for all $i$.
\end{defi}

\begin{defi}
A regular polytope $T\subset{\mathbb{R}}^{n+1}$ is a polytope
whose symmetry group ${\mathcal{S}}_T$, as a maximal subgroup of $O(n+1)$,
acts transitively on its flags in the sense for any flags $x, y\subset T$,
there exists $\phi\in {\mathcal{S}}_T$ such that $\phi(x)=y$.
\end{defi}

Regular polytopes are the high dimensional analog of regular polygons ($n=1$) and regular polyhedra ($n=2$).
A concise symbolic representation for regular polytopes was developed by Schl\"{a}fli in the 19th century,
and a slightly modified form has become standard nowadays \cite{coxeter1973regular}.
Regular polytopes can be classified according to the symmetry group ${\mathcal{S}}\subset O(n+1)$.
Regular polytopes with $k$ vertices in dimension $n\geq1$ are classified as follows,
\begin{equation}\label{e2.4}
\begin{cases}
\mbox{When } n=1, \mbox{there exist } k\mbox{-regular polytopes for each } k\in{\mathbb{N}}, k\geq3.\\
\mbox{When } n=2, \mbox{there exist only } k\mbox{-regular polytopes for } k=4, 6, 8, 12, 20.\\
\mbox{When } n=3, \mbox{there exist only } k\mbox{-regular polytopes for } k=5, 8, 16, 24, 120, 600.\\
\mbox{When } n\geq4, \mbox{there exist only } k\mbox{-regular polytopes for } k=n+2, 2n+2, 2^{n+1}.
\end{cases}
\end{equation}

\vskip5pt

\begin{example}
Let $C=\text{conv}\{(\pm1,\pm1, \pm1)\}\subset{\mathbb{R}}^3$ be a standard cube. Let ${\mathcal{S}}\subset O(3)$ be the associated symmetric group. The following statements hold.
\begin{itemize}
\item[(1)] When $a=(1,1,1)$, $\text{conv}\{\phi(a), \phi\in{\mathcal{S}}\}$ is a standard cube in ${\mathbb{R}}^3$,
which is a regular polytope.

\item[(2)] When $a=(1,0,0)$, $\text{conv}\{\phi(a), \phi\in{\mathcal{S}}\}$ is a cross-polytope in ${\mathbb{R}}^3$,
which is also a regular polytope.

\item[(3)]  When $a=(1,1,r), r\in(0,1)$, $\text{conv}\{\phi(a), \phi\in{\mathcal{S}}\}$ is a group invariant polytope
formed by cutting eight solid angles from eight vertices of the cube,
which is not a regular polytope.
\end{itemize}
\end{example}

\vspace{10pt}

\section{Existence of group  invariant solutions of \eqref{e1.2}}

Group invariant solutions have been studied in various geometric problems.
Here we consider the discrete subgroup ${\mathcal{S}}\subset O(n+1)$ that satisfies the spanning property.
By the bounded ratio condition in Section 2,
we can prove the following existence result.

\begin{theo}\label{t3.1}
 Let ${\mathcal{S}}$ be a discrete subgroup of $O(n+1)$ satisfying the spanning property.
 For  $p<-n-1$,
 equation \eqref{e1.2} admits a group invariant solution ${{h^{(p)}}}$,
 which is a maximizer of the global variational problem
 \begin{equation}\label{e3.1}
    \sup_{\Omega\in{\mathcal{K}}_{p}({\mathcal{S}})}V(\Omega)
  \end{equation}
over the group ${\mathcal{S}}$ invariant family
$$
    {\mathcal{K}}_{p}({\mathcal{S}})
     :=\bigg\{\Omega\in{\mathcal{K}}_0({\mathcal{S}})\ \Big|\ \int_{{\mathbb{S}}^{n}}fh_\Omega^p=\int_{{\mathbb{S}}^{n}}f\bigg\},
$$
where ${\mathcal{K}}_0 ({\mathcal{S}})$ is the set of group ${\mathcal{S}}$ invariant convex bodies containing the origin,
and $h_\Omega$ is the support function of $\Omega$.
\end{theo}

To prove Theorem \ref{t3.1}, it suffices to prove

\begin{lemm}\label{p3.1}
 Let ${\mathcal{S}}$ be a discrete subgroup of $O(n+1)$ satisfying the spanning property.
 For each $p<-n-1$,
the variational problem \eqref{e3.1} admits a  positive maximizer ${{h^{(p)}}}$
satisfying the equation \eqref{e1.2} up to a constant multiplication.
\end{lemm}

To prove Lemma \ref{p3.1}, we first prove the following lemmas.

\begin{lemm}\label{l3.1}
For $n\geq1$, $p<0$,
there exists a positive constant $C_{n}({\mathcal{S}})$ depending only on $n$ and ${\mathcal{S}}$, such that
\begin{equation}\label{e3.4}
      h_\Omega(x)\leq C_{n}({\mathcal{S}}), \ \ \forall\ x\in{\mathbb{S}}^{n},
\end{equation}
for each $\Omega\in{\mathcal{K}}_{p}({\mathcal{S}})$.
\end{lemm}

\begin{proof} For any $\Omega\in{\mathcal{K}}_{p}({\mathcal{S}})$,
we can assume that $\Omega\subset B_R$ (the ball with radius $R$ centered at $0$)
and that $Rb\in\partial\Omega$ for some $R>0$ and $b\in{\mathbb{S}}^{n}$.
Noting that $\Omega\supset P_b= R\, \text{conv}\{\phi(b)|\ \phi\in{\mathcal{S}}\}$ by convexity, we have
   \begin{equation}
     \max_{{\mathbb{S}}^{n}}h_\Omega/\min_{{\mathbb{S}}^{n}}h_\Omega\leq\gamma_{b}
     \leq\gamma_{max}({\mathcal{S}}),
   \end{equation}
where $\gamma_{max}({\mathcal{S}})=\sup_{a\in{\mathbb{S}}^{n}}\gamma_a$
and $\gamma_{a}$ is the ratio introduced in \eqref{e2.1}.
Thus, we obtain that
   \begin{equation}
   {\begin{split}
   \int_{{\mathbb{S}}^{n}}f
    &=\int_{{\mathbb{S}}^{n}}fh^p\leq\max_{{\mathbb{S}}^{n}} h^p\int_{{\mathbb{S}}^{n}}f  \\
    &\le \bigg(\frac{\max_{{\mathbb{S}}^{n}}h}{\gamma_{max}({\mathcal{S}})}\bigg)^{p}\int_{{\mathbb{S}}^{n}}f.
    \end{split}}
   \end{equation}
This gives exactly the desired estimate \eqref{e3.4} for negative $p$.
\end{proof}

\begin{lemm}\label{l3.2}
For $n\geq1$, $p<-n$,
and a group ${\mathcal{S}}$ invariant positive function $f$ on $\mathbb{S}^{n}$,
there exists a positive constant $C_{n}$ depending only on $n$ such that
\begin{equation}\label{e3.5}
 h(x)\geq \Big[(-p)^n\cdot C_{n}\cdot \gamma_{max}({\mathcal{S}})\cdot
     \frac{\max_{{\mathbb{S}}^{n}}f}{\min_{{\mathbb{S}}^{n}}f}\Big]^\frac{1}{n+p}, \ \ \forall \ x\in{\mathbb{S}}^{n},
\end{equation}
for any support function $h$ with the associated convex body in ${\mathcal{K}}_{p}({\mathcal{S}})$.
\end{lemm}

\begin{proof}
Denote $M=\max_{{\mathbb{S}}^{n}}h$ and $m= \min_{{\mathbb{S}}^{n}}h$.
Assume that  $m=h(x_0)$ for the south polar $x_0$.
By the convexity of $h$, we have
\begin{equation*}
    h(x)\leq m+M\sin\theta, \ \ \forall x\in{\mathbb{S}}^{n},
\end{equation*}
where $\theta= \text{dist}(x,x_0)$ is the geodesic distance from $x_0$ to $x$. Hence,
\begin{eqnarray*}
   \int_{{\mathbb{S}}^{n}}fh^p
   &\geq& C_1\min_{{\mathbb{S}}^{n}}f\int^{\pi/2}_0\frac{\theta^{n-1} d\theta}{(m+M\theta)^{-p}}\\
     &=&   C_1M^{-n}m^{p+n}\min_{{\mathbb{S}}^{n}}f\int^{M\pi/(2m)}_0\frac{t^{n-1}}{(1+t)^{-p}}dt
\end{eqnarray*}
for a positive constant $C_1$ depending only on $n$.
By taking some large constant $C_2>0$, noting that
 $$
  \int^{M\pi/(2m)}_0\frac{t^{n-1}}{(1+t)^{-p}}dt\geq\int^{1/(-pC_2)}_0\frac{t^{n-1}}{(1+t)^{-p}}dt\geq C_3^{-1}(-p)^{-n}
 $$
holds for another positive constant $C_3$ depending only on $n$,
inequality \eqref{e3.5} then follows from Lemma \ref{l3.1} when $p<-n$.
\end{proof}

As a corollary of Lemma \ref{l3.2}, we have the following result.

\begin{coro}\label{c3.1}
 Let $\Omega^{(p)}$ be the solution obtained in Lemma \ref{p3.1} with the support function $h^{(p)}$, we have
   \begin{equation}\label{e3.8}
     \lim_{p\to-\infty}\min_{{\mathbb{S}}^n}h^{(p)}=1.
   \end{equation}
\end{coro}

\begin{proof} By Lemma \ref{l3.2}, sending $p$ to negative infinity in \eqref{e3.5}, we conclude that
$$
\lim_{p\to-\infty}\min_{x\in{\mathbb{S}}^{n}}{{h^{(p)}}}(x)\geq1.
$$
On the other hand, evaluating the solution ${{h^{(p)}}}$ at its minimal point on the sphere ${\mathbb{S}}^{n}$, it follows from \eqref{e1.2} by the maximum principle that
  $$
   \min_{x\in{\mathbb{S}}^{n}}{{h^{(p)}}}(x)\leq\max_{x\in{\mathbb{S}}^{n}}f^{1/(n-p)}(x).
  $$
Sending $p$ to negative infinity, the desired identity \eqref{e3.8} follows.
\end{proof}

When {\it a-priori} lower-upper bounds of functions $h\in{\mathcal{K}}_{p}({\mathcal{S}})$ have been derived in Lemmas \ref{l3.1} and \ref{l3.2}, by Blaschke's selection theorem, a maximizing sequence $\Omega_j\in{\mathcal{K}}_{p}({\mathcal{S}})$ converges to a limiting convex body ${{\Omega^{(p)}}}\in{\mathcal{K}}_{p}({\mathcal{S}})$. Due to the uniform convergence of the support functions $h_j$ to ${{h^{(p)}}}$, we see that ${{\Omega^{(p)}}}$ is actually a maximizer of the variational problem \eqref{e3.1}. Consequently, we can prove Lemma \ref{p3.1}.

\begin{proof}[Proof of Lemma \ref{p3.1}]
To show Lemma \ref{p3.1}, it remains to verify that ${{h^{(p)}}}$ satisfies the Euler-Lagrange equation \eqref{e1.2}.
This can be achieved by adapting the argument in (\cite{CCL21}, Section 3) (see also \cite{CW06}, Section 5) with slight modifications. First, by calculating the first variation along the Wulff shape, it follows that the Monge-Amp\`{e}re measure associated to the solution is bounded from above. By the duality of the polar bodies, we have also that the Monge-Amp\`{e}re measure is bounded from below. Then, using the work of Caffarelli in \cite{caffarelli1990localization}, we can demonstrate the strict convexity and $C^{1,\alpha}$ regularity of the solution. Consequently, we can show that the maximizer satisfies that
     \begin{equation}\label{e3.7}
       \int_{{\mathbb{S}}^n}\big[\det(\nabla^2h+hI)-fh^{p-1}\big]g=0
     \end{equation}
  for all group invariant function $g$, we thus reach the conclusion that $h=h^{(p)}$ satisfies the Euler-Lagrange equation \eqref{e1.2}. Actually, one needs only to consider group invariant Hilbert space
    $$
     {\mathbb{H}}({\mathcal{S}})\equiv\{g\in L^2({\mathbb{S}}^n)|\ g\mbox{ is group invairant}\}
    $$
  and use the fact that the orthonormal subspace of ${\mathbb{H}}({\mathcal{S}})$ itself must be identical to $\{0\}$.
  Then, \eqref{e3.7} means that the group invariant function $\widetilde{f}:=\det(\nabla^2h+hI)-fh^{p-1}$
  belongs to the orthonormal subspace of ${\mathbb{H}}({\mathcal{S}})$, and hence must be identical to zero.

If $f$ is H\"older continuous,  by  \cite{caffarelli1990interior} we conclude that $h^{(p)}$ is $C^{2,\alpha}$ smooth.
Higher regularity can be obtained by iteration using Schauder's estimate.
This completes the proof of Lemma \ref{p3.1}.
\end{proof}

\vspace{10pt}

\section{ Limiting extremal bodies}

By Theorem \ref{t3.1} and its proof, we have

\begin{theo}\label{t4.1}
 Let ${\mathcal{S}}$ be a discrete subgroup of $O(n+1)$ satisfying the spanning property.
 For each $p<-n-1$, there holds
   \begin{equation}\label{4.1}
      \mathcal{F}_{p}(\Omega)\leq C(n,p,f,{\mathcal{S}}), \ \ \forall\ \Omega\in{\mathcal{K}}_0({\mathcal{S}}).
   \end{equation}
with the best constant
$$C(n,p,f,{\mathcal{S}}):=\sup_{\Omega\in{\mathcal{K}}_{p}({\mathcal{S}})}V(\Omega) .$$
Moreover,
the inequality \eqref{4.1} becomes equality at some smooth group invariant solution ${{h^{(p)}}}$ of \eqref{e1.2},
corresponding to a group invariant convex body ${{\Omega^{(p)}}}\in{\mathcal{K}}_0({\mathcal{S}})$.
\end{theo}

\begin{proof} For any convex body $\Omega\in{\mathcal{K}}_0({\mathcal{S}})$ with support function $h_\Omega$,
we normalize it to $\widetilde{\Omega}$ with support function
  $$
   \widetilde{h}:= h_\Omega/\lambda\in{\mathcal{K}}_0({\mathcal{S}}),
    \ \ \lambda=\left(\int_{{\mathbb{S}}^{n}}fh_\Omega^p{\Big/}
   \int_{{\mathbb{S}}^{n}}f\right)^{1/p}.
  $$
Then we have
  $$
   \mathcal{F}_{p}(\Omega)=V(\widetilde{\Omega})\leq C(n,p,f,{\mathcal{S}}):=\sup_{\Omega\in{\mathcal{K}}_{p}({\mathcal{S}})}V(\Omega).
  $$
Moreover, the equality holds at multiples of $h^{(p)}$. This completes the proof.
\end{proof}

It is worth noting that when the Santal\'{o} center of $\Omega$ is the origin and when $f=1$,
our functional $\mathcal{F}_{-n-1}(\cdot)$ is exactly  Mahler's volume \cite{Sch14}.
By Lemmas \ref{l3.1} and \ref{l3.2}, the sequence of maximizers ${{h^{(p)}}}$ of the functional is {\it a-priori} bounded
in $Lip({\mathbb{S}}^{n})$ by convexity.
Hence, along a subsequence, the convex body ${{\Omega^{(p)}}}$ converges to
a limiting convex body $\Omega^{(-\infty)}\in{\mathcal{K}}_0({\mathcal{S}})$.
We have the following characterization of $\Omega^{(-\infty)}$.

\begin{lemm}\label{l4.1}
Let ${\mathcal{S}}$ be a discrete subgroup of $O(n+1)$ satisfying the spanning property. The best constant $C(n,p,f,{\mathcal{S}})$ of the functional $\mathcal{F}_{p}(\cdot)$ converges to the best constant $C(n,-\infty,{\mathcal{S}})$ of the limiting functional
     $$
      \mathcal{F}_{-\infty}(\Omega):= V(\Omega)/\min_{{\mathbb{S}}^{n}}h_\Omega^{n+1},
                       \ \ \Omega\in{\mathcal{K}}_0({\mathcal{S}}).
     $$
   Moreover, $\Omega^{(-\infty)}$ is exactly a maximizer of the functional $\mathcal{F}_{-\infty}(\cdot)$.
\end{lemm}

\begin{proof}
For each $p<-n$, we have
   \begin{equation}\label{e4.2}
     \mathcal{F}_{p}(\Omega)\leq \mathcal{F}_{-\infty}(\Omega)\Rightarrow C(n,p,f,{\mathcal{S}})\leq C(n,-\infty,{\mathcal{S}}).
   \end{equation}
Conversely, for each $\Omega\in{\mathcal{K}}_0({\mathcal{S}})$, by the convergence property of $L^p$-norm to $L^\infty$-norm, we have
  \begin{equation}\label{e4.3}
   \mathcal{F}_{-\infty}(\Omega)=\lim_{p\to-\infty}\mathcal{F}_{p}(\Omega)\Rightarrow C(n,-\infty,{\mathcal{S}})\leq\lim_{p\to-\infty}C(n,p,f,{\mathcal{S}}).
  \end{equation}
Combining \eqref{e4.3} with \eqref{e4.2} yields
   \begin{equation}\label{e4.4}
     C(n,-\infty,{\mathcal{S}})=\lim_{p\to-\infty}C(n,p,f,{\mathcal{S}}).
   \end{equation}
On the other hand,
by the uniform convergence of $1/{{h^{(p)}}}$ to $1/h^{(-\infty)}$,
 and the condition \eqref{fcc}, we have
  \begin{eqnarray*}
   &&\left[\frac{1}{(n+1)\omega_{n+1}}
   \int_{{\mathbb{S}}^{n}}f|1/{{h^{(p)}}}-1/h^{(-\infty)}|^{-p}\right]^{-1/p}\\
   &&\hskip+45pt\leq\max_{{\mathbb{S}}^{n}}f^{-1/p}||1/{{h^{(p)}}}-1/h^{(-\infty)}||_{L^\infty}
   =o_p(1),
  \end{eqnarray*}
where the infinitesimal $o_p(1)$ is small as long as $p$ is negative large. As a result,
 \begin{align*}
   \mathcal{F}_{-\infty}(\Omega^{(-\infty)})&=\mathcal{F}_{p}(\Omega^{(-\infty)})+o_p(1)
   =\mathcal{F}_{p}({{\Omega^{(p)}}})+o_p(1)\\
    &=C(n,p,f,{\mathcal{S}})+o_p(1)=C(n,-\infty,{\mathcal{S}})+o_p(1)
 \end{align*}
by \eqref{e4.4}. This implies that $\Omega^{(-\infty)}$ is exactly a maxmizer of the functional $\mathcal{F}_{-\infty}(\cdot)$.
\end{proof}

The next lemma rules out the trivial possibility of the limiting sphere ${\mathbb{S}}^{n}$.

\begin{lemm}\label{l4.2}
  Let $P\subset {B_1}$ be any group ${\mathcal{S}}$ invariant polytope, then
    \begin{equation}\label{e4.5}
     \mathcal{F}_{-\infty}(P)>\mathcal{F}_{-\infty}(B_1).
    \end{equation}
 Consequently, the limiting shape $\Omega^{(-\infty)}$ of ${{\Omega^{(p)}}}$,
 when $p\to-\infty$,  cannot be a sphere.
\end{lemm}

\begin{proof}
Let $P$ be a group ${\mathcal{S}}$ invariant polytope.
Setting $m = \min h_{P}$ to be the inner-radius of $P$, we have $m {B_1}$ as a smaller ball contained inside and tangential to $P$. Hence, the desired inequality \eqref{e4.5} follows from
  $$
   \frac{\mathcal{F}_{-\infty}(P)}{\mathcal{F}_{-\infty}({B_1})}
   =\frac{\mathcal{F}_{-\infty}(P)}{\mathcal{F}_{-\infty}(m {B_1})}
   =\frac{V(P)/m^{n+1}}{V(m {B_1})/m^{n+1}}>1.
  $$
By Lemma \ref{l4.1}, the limiting shape $\Omega^{(-\infty)}$ is a maximizer of the functional $\mathcal{F}_{-\infty}(\cdot)$. Thus, it can not be a ball by \eqref{e4.5}.
\end{proof}

The next lemma implies that the maximizer of the limiting functional must be a polytope.

\begin{lemm}\label{l4.3}
  For any non-polytope group invariant convex body $\Omega\in{\mathcal{K}}_0({\mathcal{S}})$, there exists a group ${\mathcal{S}}$ invariant polytope $P$ containing $\Omega$ such that
   \begin{equation}
       \mathcal{F}_{-\infty}(\Omega)<\mathcal{F}_{-\infty}(P).
     \end{equation}
\end{lemm}
\begin{proof}
Let $x_0\in\partial\Omega$ and $B_{r_\Omega}\subset\Omega$ for the inner-radius $r_\Omega$ of $\Omega$, we denote $H(x_0)$ to be the tangent hyperplane of $\partial\Omega$ at $x_0$, and let $H^-(x_0)$ be the half space of $H(x_0)$ containing $\Omega$. Define
$$
P_{x_0}=\cap_{\phi\in{\mathcal{S}}}\phi(H^-(x_0)).
$$
Obviously, $P_{x_0}$ is a group ${\mathcal{S}}$ invariant polytope containing $\Omega$, then
\begin{eqnarray*}
    \frac{\mathcal{F}_{-\infty}(P_{x_0})}{\mathcal{F}_{-\infty}(\Omega)}
    =\frac{V(P_{x_0})}{V(\Omega)}>1,
  \end{eqnarray*}
as long as $\Omega$ is not a polytope.
\end{proof}

Combining with Theorems \ref{t3.1} and \ref{t4.1}, along with Lemma \ref{l4.3}, we obtain the following theorem, which characterizes the limiting extremal body as a maximal group-invariant polytope.

\begin{theo}\label{t4.2}
 Let ${\mathcal{S}}$ be a discrete subgroup of $O(n+1)$ satisfying the spanning property.
 The maximizer $\Omega^{(-\infty)}$ of the functional $\mathcal{F}_{-\infty}(\cdot)$ must be a group invariant polytope,
 after scaling such that the volume $|\Omega^{(-\infty)}|>0$.
\end{theo}
\begin{proof}[Proof of Theorem \ref{t1.1}]
Theorem \ref{t1.1} follows from Theorem \ref{t4.2} and Corollary \ref{c3.1}.
\end{proof}

\vspace{10pt}
\section{Local maximizer of the functional and proof of Theorem \ref{t1.2}}

In this section, we use $T\subset \mathbb{R}^{n+1}$ to denote a regular polytope that is tangential to the unit sphere ${\mathbb{S}}^n$, and $\mathcal{S}_T\subset O(n+1)$ to represent the symmetry group of $T$.
Denote
   $$
    {\mathcal{N}}_\delta(T)=\big\{\Omega\in{\mathcal{K}}_0({\mathcal{S}}_T)\ \big|\ dist(\Omega,T)<\delta\big\}
   $$
the $\delta$-neighborhood of $T$ for a positive constant $\delta$ and the Hausdorff distance $dist(\cdot,\cdot)$,
and denote $\p  {\mathcal{N}}_\delta(T)$ the boundary of $ {\mathcal{N}}_\delta(T)$.
Denote also the normalized group invariant family
    $$
     {\mathcal{K}}_{min}({\mathcal{S}}_T)=\big\{\Omega\in{\mathcal{K}}_0({\mathcal{S}}_T)
     \ \big|\ \inf_{{\mathbb{S}}^n}h_\Omega=1\big\},
    $$
and the normalized $\delta$-neighborhood
   $$
    {\mathcal{K}}_{min}(\delta,{\mathcal{S}}_T)={\mathcal{K}}_{min}({\mathcal{S}}_T)\cap{\mathcal{N}}_\delta(T)
   $$
of $T$. Then, we can prove that $T$ has strictly maximal volume in the normalized family ${\mathcal{K}}_{min}(\delta,{\mathcal{S}}_T)$ for some $\delta=\delta_T>0$ as follows.
\begin{lemm}\label{l5.1}
  For any regular polytope $T\subset{\mathbb{R}}^{n+1}$ tangential to the unit sphere ${\mathbb{S}}^n$, there exists a small constant $\delta=\delta_T>0$ such that
     \begin{equation}\label{e5.1}
       V(T)>V(\Omega), \ \ \forall \ \Omega\in{\mathcal{K}}_{min}(\delta,{\mathcal{S}}_T)\setminus\{T\}.
     \end{equation}
\end{lemm}

\begin{proof}
\emph{Step 1}. Subdivision of a regular polytope.

Let $T\subset{\mathbb{R}}^{n+1}$ be a regular polytope that is tangential to the unit sphere ${\mathbb{S}}^n$.
By the high symmetry of $T$, each $\kappa$-dimensional face is also a regular polytope, where $\kappa\in[2,n]$ and $\kappa\in{\mathbb{N}}$.
We will use this high symmetry to decompose $T$ into as many congruent pieces as possible.
Let  $O_{n+1}$, the origin of ${\mathbb{R}}^{n+1}$, be the centroid of $T$.
For  an $n$-dimensional face $F_n$ of $T$, we denote by $O_n$ the centroid of $F_n$.
Let $F_{n-1}$ be one of the $(n-1)$-dimensional face of $F_n$, we denote its centroid by $O_{n-1}$.
The bootstrap procedure will produce a sequence of faces
  $$
   T=F_{n+1}\rightarrow F_n\rightarrow F_{n-1}\rightarrow \cdots \rightarrow F_{1} \rightarrow F_{0}
  $$
with lower dimensions, and yield a sequence of corresponding centroids
   $$
   O_{n+1}\rightarrow O_n\rightarrow O_{n-1}\rightarrow \cdots\rightarrow  O_{1} \rightarrow  O_{0}.
   $$
Note that
   \begin{equation}\label{e5.2}
     O_{j}O_{j-1}\perp \text{span}\{O_i, i=0,1,\cdots,j-1\}, \ \ \forall j=n+1,n,\cdots,1.
   \end{equation}
Since $O_{j}O_{j-1}\perp F_{j-1}, \forall j=n+1,n,\cdots,1$,
one may assume that $O_jO_{j-1}$ is parallell to the $x_j$-axis by rotation
and set $R_{n+1}=1, R_j=|O_jO_{j-1}|, \forall j=n,n-1,\cdots,2,1$.
Define $\Omega_j:=\text{conv}(\cup_{i=0}^jO_i), j=0,1,\cdots,n$, we have the recursive formula
   \begin{equation}\label{e5.3}
     \Omega_0=O_0, \ \
    \Omega_j=\text{conv}\{O_j,\Omega_{j-1}\},\ \   j=1,2,\cdots,n.
   \end{equation}
Then, $F_n$ can be decomposed into small pieces that are congruent to $\Omega_n$\footnote{Sometimes, we may regard $\Omega_l\subset\{z\in{\mathbb{R}}^{n+1}| z_{n+1}=R_{n+1}, z_n=R_n, \cdots, z_{l+1}=R_{l+1}, z_j\in[0,R_j], \forall j\in[1,l]\}$ as a subdomain of ${\mathbb{R}}^l$ for $l=1,2,\cdots,n$ by simply omitting the coordinates $z_k=R_k$ for $k=n+1, n, \cdots, l+1$. For example, $\Omega_n\subset\{\bar{z}\in\mathbb{R}^{n}\}$ for each $z=(\bar{z},z_{n+1})\in\mathbb{R}^{n+1}$.}.
Let $w_0=(\overline{w}_0,1)\in \Omega_n\subset{\mathbb{R}}^{n+1}$ be the point of the piece $\Omega_n$ defined above with projection
  $$
   y_0=w_{0}/|w_{0}|\in D_{n}\subset{\mathbb{S}}^n,
  $$
satisfying $h'(y_0)=1$ for support function $h'$ of another normalized group invariant convex body $\Omega'\in \mathcal{K}_{min}({\mathcal{S}}_T)$.

{\centering
$\hskip10pt \includegraphics[height=60mm]{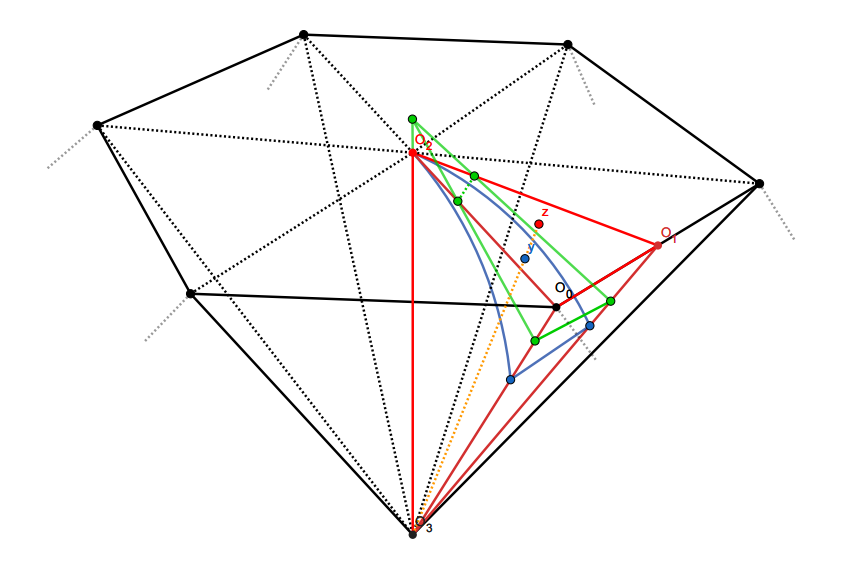} $}

\vskip-175pt

{\centering
$\hskip290pt \includegraphics[height=60mm]{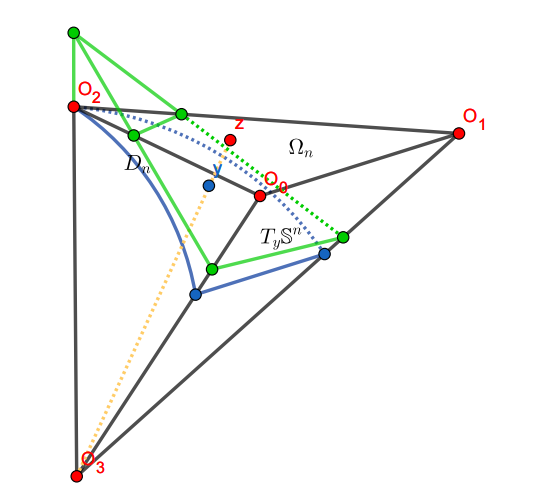} $}

\vskip10pt

 \qquad Figure: subdivision of a polyhedron tangential to the unit sphere ${\mathbb{S}}^2$ in $\mathbb{R}^3$.

\vskip10pt

\noindent\emph{Step 2.} Claim $\sharp1$: \emph{The regular polytope $T$ has local strictly maximal volume on the normalized group invariant family ${\mathcal{K}}_{min}(\delta,{\mathcal{S}}_T)$ if the following integral function
  \begin{equation}\label{e5.4}
   \overline{V}(\overline{w}_0)=\int_{\Omega_n}\left(1-\bigg|\frac{\sqrt{1+|\overline{w}_0|^2}}{1+\langle \overline{z},\overline{w}_0\rangle}\bigg|^{n+1}\right)d\overline{z}
  \end{equation}
is positive for all $\overline{w}_0\in\Omega_n\setminus \{0\}$ closing to $0$. (See footnote 1 for $\Omega_n$.)}
\begin{proof}[Proof of claim $\sharp1$]
At first, the hyperplane $H_1$ tangential to the unit sphere ${\mathbb{S}}^n$ at $O_n=e_{n+1}$ is given by the radial function $\rho_1(y)= 1/y_{n+1}$ for $y=(y_1,y_2,\ldots, y_{n+1})\in{\mathbb{S}}^n$. While, the hyperplane $H_2$ tangential to the $\mathbb{S}^n$ at $y_0$ is given by the radial function $\rho_2(y)=1/\langle y,y_0\rangle$. Therefore, we obtain that
    \begin{equation}\label{e5.5}
     V_{D_n}(H^-_1)-V_{D_n}(H^-_2)=\frac{1}{n+1}\int_{D_n}\left(\frac{1}{y_{n+1}^{n+1}}-\frac{1}{|\langle y,y_0\rangle|^{n+1}}\right)d\sigma,
    \end{equation}
where $V_{D_n}(H^-)$ represents the volume of the solid enclosed by the hyperplane $H$ and the infinite cone formed by connecting the origin to $D_n$. Note that for each normalized group invariant convex body $\Omega'\in{\mathcal{K}}_{min}(\delta,{\mathcal{S}}_T)$, we also have
    \begin{equation}\label{e5.6}
      V_{D_n}(T)=V_{D_n}(H^-_1), \ \ V_{D_n}(\Omega')\leq V_{D_n}(H^-_2).
    \end{equation}
Regarding the coordinates $\overline{z}$ in the hyperplane $H_1=\{z=(\overline{z},z_{n+1})\in{\mathbb{R}}^{n+1}|\ z_{n+1}=1\}$ as the local coordinates of ${\mathbb{S}}^n$ by projection $y=z/|z|$, we have the metric
  $$
    g_{ij}=\frac{\partial y}{\partial z_i}\cdot\frac{\partial y}{\partial z_j}=\frac{\delta_{ij}}{|z|^2}-\frac{z_iz_j|\overline{z}|^2}{|z|^6}
  $$
since $\frac{\partial y}{\partial z_i}=\left(\frac{e_i}{|z|}-\frac{\bar{z}z_i}{|z|^3},-\frac{z_i}{|z|^3}\right)$,
and thus the volume element is
   $$
    d\sigma=\sqrt{\det(g)}d\overline{z}=(1+|\overline{z}|^2)^{-\frac{n+1}{2}}d\overline{z}.
   $$
Therefore, the volume difference function
    $${\begin{split}
      \widetilde{V}(y_0): & =(n+1)(V_{D_n}(H^-_1)-V_{D_n}(H^-_2))\\
      & =\int_{\Omega_n}\left(1-\frac{1}{|\langle z,y_0\rangle|^{n+1}}\right)d\overline{z}
      \end{split}}
    $$
holds for $y_0\in D_n$. (See footnote 1 in page 14 for $\Omega_n$.) Letting $y_{0,n+1}$ be the $(n+1)$-th coordinates of $y_0$ and setting $w_0=(\overline{w}_0,1)=\frac{y_0}{y_{0,n+1}}\in\Omega_n\subset\mathbb{R}^{n+1}$, there holds
    \begin{equation}\label{e5.7}
     \widetilde{V}(y_0)=\overline{V}(\overline{w}_0):=
     \int_{\Omega_n}\left(1-\bigg|\frac{\sqrt{1+|\overline{w}_0|^2}}{1+\langle \overline{z},\overline{w}_0\rangle}\bigg|^{n+1}\right)d\overline{z},\ \ \overline{w}_0\in\Omega_n,
    \end{equation}
due to $y_0=(\overline{w}_0,1)/|(\overline{w}_0,1)|$.
This, together with \eqref{e5.6}, completes the proof of claim $\sharp1$.
\end{proof}
\noindent\emph{Step 3.} Claim $\sharp 2$: \emph{The regular polytope $T$ has locally strictly maximal volume on the normalized group invariant family ${\mathcal{K}}_{min}(\delta,{\mathcal{S}}_T)$ for some $\delta>0$. More precisely, there exists a positive constant $\delta=\delta_{T}$ such that for any $\Omega'\in{\mathcal{K}}_{min}(\delta,{\mathcal{S}}_T)\setminus\{T\}$, there holds
    \begin{equation}\label{e5.8}
      V(\Omega')<V(T).
    \end{equation}}
\begin{proof}[Proof of claim $\sharp2$]The proof of the claim is equivalent to verifying that: If $w_0=(\overline{w}_0,1)\in\Omega_n$ approaches $O_n=e_{n+1}$ without equaling $O_n$, then the volume difference function $\overline{V}(\overline{w}_0)$  in \eqref{e5.4} is positive. Along the direction segment $I=\{O_n+ra\in\Omega_n| r\geq0, a\in{\mathbb{R}}^n_+\}$ \footnote{${\mathbb{R}}^{j}_+:=\{z\in{\mathbb{R}}^{j}| \Sigma _{i=1}^jz_i> 0 \ \text{and}\ z_i\ge0\ \text{for each}\ i=1,2,\cdots, j\}, \ \ j=1,2,\cdots,n+1.$}, the volume difference function can be simplified by for $\overline{w}_0=ra$,
    $$
     V(r,a):=\overline{V}(\overline{w}_0)=\int_{\Omega_n}\bigg(1-\bigg|\frac{\sqrt{1+r^2|a|^2}}{1+r\langle \overline{z},a\rangle}\bigg|^{n+1}\bigg)d\overline{z}.
    $$
 Taking the first differentiation in $r$ and then evaluating at $r=0$, we get
    \begin{equation}
     \frac{d}{dr}\bigg|_{r=0}V(r,a)=(n+1)\int_{\Omega_n}\langle\overline{z},a\rangle d\overline{z}>0, \ \ \forall a\in{\mathbb{R}}^{n}_+
    \end{equation}
using the fact that $\Omega_{n}\subset{\mathbb{R}}^{n}_+$ (See footnote 1 in page 14 for $\Omega_n$).
This, together with Claim $\sharp1$, yields $V(r,a)>V(0,a)=0$, and hence the desired inequality \eqref{e5.8} holds.
\end{proof}
Then Lemma \ref{l5.1} follows directly from claim $\sharp 2$.
\end{proof}

For the given regular polytope $T\subset{\mathbb{R}}^{n+1}$ tangential to ${\mathbb{S}}^n$, define the group ${\mathcal{S}}_T$ invariant family
   $$
    {\mathcal{K}}_{p,f}({\mathcal{S}}_T)\equiv\Big\{\Omega\in{\mathcal{K}}_0({\mathcal{S}}_T)
    \ \big|\ \int_{{\mathbb{S}}^{n}}fh_\Omega^p=\int_{{\mathbb{S}}^{n}}fh_T^p\Big\}, \ \ p<0.
   $$
As shown in Lemma \ref{l5.1}, $T$ is a local strict maximizer of the functional ${\mathcal{F}}_{-\infty}(\cdot)$ on the normalized family ${\mathcal{K}}_{min}(\delta,{\mathcal{S}}_T)$ for some $\delta>0$. Hence, there exists a positive constant $\delta=\delta_{T}$ such that
   \begin{equation}\label{e5.10}
      {\mathcal{F}}_{-\infty}(\Omega)<{\mathcal{F}}_{-\infty}(T), \ \ \forall\ \Omega\in{\mathcal{K}}_{min}(\delta,{\mathcal{S}}_T)\setminus T.
   \end{equation}
We now consider the local variational scheme
    \begin{equation}\label{e5.11}
       \sup_{\Omega\in{\mathcal{K}}_{p,f}(\delta,{\mathcal{S}}_T)}V(\Omega),
    \end{equation}
where
$\mathcal{K}_{p,f}(\delta,{\mathcal{S}}_T):={\mathcal{K}}_{p,f}({\mathcal{S}}_T) \cap{\mathcal{N}}_\delta(T)$ and $ \partial\mathcal{K}_{p,f}(\delta,{\mathcal{S}}_T):={\mathcal{K}}_{p,f}({\mathcal{S}}_T)
\cap\partial{\mathcal{N}}_\delta(T)$.
Note that $\mathcal{K}_{min}(\delta,{\mathcal{S}}_T)=\cap_{p'<-n-1}\cup_{p\in(-\infty,p')}\mathcal{K}_{p,f}(\delta,{\mathcal{S}}_T)$,
i.e., $\mathcal{K}_{p,f}(\delta,{\mathcal{S}}_T)\rightarrow \mathcal{K}_{min}(\delta,{\mathcal{S}}_T)$
 as $p\rightarrow -\infty$.
In fact, the convergence of ${\mathcal{K}}_{p,f}(\delta,{\mathcal{S}}_T)$ to the limiting family
   \begin{equation}\label{e-family}
    {\mathcal{K}}_{-\infty,f}(\delta,{\mathcal{S}}_T)
     =\Big\{\Omega\in{\mathcal{K}}_0({\mathcal{S}}_T)  \ \big|\
       \inf_{{\mathbb{S}}^n}h_\Omega=\inf_{{\mathbb{S}}^n}h_T=1\Big\}\cap{\mathcal{N}}_\delta(T)
         ={\mathcal{K}}_{min}(\delta,{\mathcal{S}}_T)
   \end{equation}
under the Hausdorff distance. To clarify this convergence, we mean that for each $\Om^{(-\infty)}\in{\mathcal{K}}_{-\infty,f}(\delta,{\mathcal{S}}_T)$, there exists a sequence of $\Om^{(p)}\in{\mathcal{K}}_{p,f}(\delta,{\mathcal{S}}_T)$ converges to $\Om^{(-\infty)}$ under Hausdorff distance. Vice versa, for any sequence of $\Om^{(p)}\in{\mathcal{K}}_{p,f}(\delta,{\mathcal{S}}_T)$, we have $\Om^{(p)}$ sub-converges to some $\Om^{(-\infty)}\in{\mathcal{K}}_{-\infty,f}(\delta,{\mathcal{S}}_T)$.

We now show the existence of maximizer $\Omega^{(p)}\in{\mathcal{K}}_{p,f}(\delta,{\mathcal{S}}_T)$ of \eqref{e5.11} for negative large $p$, whose support function $h^{(p)}$ satisfies the Euler-Lagrange equation up to a constant multiplication.
The main ingredient of the proof is to show that the maximizer $\Omega^{(p)}$ of \eqref{e5.11} stays away from the boundary of ${\mathcal{N}}_\delta(T)$, that is, $\sup_{\Omega\in\partial{\mathcal{K}}_{p,f}(\delta,{\mathcal{S}}_T)}V(\Omega)
      <\sup_{\Omega\in{\mathcal{K}}_{p,f}(\delta,{\mathcal{S}}_T)}
      V(\Omega)=V(\Omega^{(p)})$ and $\Omega^{(p)}\not\in \partial {\mathcal{N}}_\delta(T)$.

\begin{lemm}\label{l5.2}
 Let $T\subset{\mathbb{R}}^{n+1}$ be a regular polytope tangential to the unit sphere ${\mathbb{S}}^n$.
 For each $p<-n-1$,
the variational problem \eqref{e5.11} admits a maximizer $\Omega^{(p)}$ staying away from the boundary of ${\mathcal{N}}_\delta(T)$ for a small $\delta=\delta_T>0$.
\end{lemm}

\begin{proof}
For simplicity, let $\mathcal{K}(\delta,{\mathcal{S}}_T)$ denote the union of ${\mathcal{K}}_{p,f}(\delta,{\mathcal{S}}_T)$ for $p<-n-1$, that is, $\mathcal{K}(\delta,{\mathcal{S}}_T)
=\cup_{p<-n-1}{\mathcal{K}}_{p,f}(\delta,{\mathcal{S}}_T)$. By {\it a-priori} Lipschitz bound of convex bodies $\Omega\in{\mathcal{K}}(\delta,{\mathcal{S}}_T)$, the boundary $\partial{\mathcal{K}}(\delta,{\mathcal{S}}_T)$ is a sequential compact set under the Hausdorff distance. Noting that the limiting functional ${\mathcal{F}}_{-\infty}(\cdot)$ is continuous under the Hausdorff distance, it follows from \eqref{e5.10} that for some $\sigma>0$,
    \begin{equation}\label{e5.12}
      \sup_{\Omega\in\partial{\mathcal{K}}_{min}(\delta,{\mathcal{S}}_T)}
      {\mathcal{F}}_{-\infty}(\Omega)\leq{\mathcal{F}}_{-\infty}(T)-\sigma.
    \end{equation}
Noting that for fixed $\Omega$, we have the convergence
   \begin{equation}\label{e5.13}
     \lim_{p\to-\infty}{\mathcal{F}}_{p}(\Omega)={\mathcal{F}}_{-\infty}(\Omega).
   \end{equation}

We claim that the functionals ${\mathcal{F}}_p(\cdot)$ are equi-continuous for negative large $p$. Actually, by Minkowski's inequality for the $f$-weighted $L^q_f({\S}^n)$-space defined by the norm
    $$
    \|h\|_{L^q_f({\S}^n)}=\left(\int_{{\S}^n}fh^q\right)^{1/q}, \ \ \forall h\in L^q_f({\S}^n)
    $$
and a positive constant $q>1$, we have for some constant $C>0$,
\begin{eqnarray*}
\left|\left(\int_{{\S}^n}fh_{\Omega}^p\right)^{-1/p}
-\int_{{\S}^n}fh_{\Omega'}^p\bigg)^{-1/p}\right|
&=&\left|\|h^{-1}_\Omega\|_{L^{|p|}_f({\S}^n)}-\|h^{-1}_{\Omega'}\|_{L^{|p|}_f({\S}^n)}
\right|\\
&\leq&\|h^{-1}_\Omega-h^{-1}_{\Omega'}\|_{L^{|p|}_f({\S}^n)}\\
&\leq& C\|h_\Omega-h_{\Omega'}\|_{L^{|p|}_f({\S}^n)}.
     \end{eqnarray*}
Thus, the equi-continuity of the functional ${\mathcal{F}}_p(\cdot)$ follows from the definition \eqref{f1.3}.

Combining with the sequential compactness of $\partial{\mathcal{K}}(\delta,{\mathcal{S}}_T)$ and the equi-continuity of the functionals ${\mathcal{F}}_p(\cdot)$ for negative large $p$, we have
\begin{equation}\label{e5.14}
\lim_{p\to-\infty}\sup_{\Omega\in\partial{\mathcal{K}}(\delta,{\mathcal{S}}_T)}
|{\mathcal{F}}_{p}(\Omega)-{\mathcal{F}}_{-\infty}(\Omega)|=0,
\end{equation}
in particular,
$$
\lim_{p\to-\infty}|{\mathcal{F}}_{p}(T)-{\mathcal{F}}_{-\infty}(T)|=0.
$$
For short, we write
$$
C(n,p,f,{\mathcal{S}_T})=\sup_{\Omega\in{\mathcal{K}}_{p,f}(\delta,{\mathcal{S}}_T)}V(\Omega).
$$
By the compactness of ${\mathcal{K}}(\delta,{\mathcal{S}}_T)$ and formula \eqref{e-family}, we have the convergence
    \begin{equation}\label{e5.15}
\lim_{p\to-\infty}C(n,p,f,{\mathcal{S}_T})=C(n,-\infty,f,{\mathcal{S}_T})
=\sup_{\Omega\in{\mathcal{K}}_{min}(\delta,{\mathcal{S}}_T)}V(\Omega).
    \end{equation}
 Noting that $C(n,-\infty,f,{\mathcal{S}_T})=\sup_{\Omega\in{\mathcal{K}}_{min}(\delta,{\mathcal{S}}_T)}V(\Omega)
=V(T)=\mathcal{F}_{-\infty}(T)$ by Lemma \ref{l5.1}, it follows from \eqref{e5.12}-\eqref{e5.15} that
   \begin{eqnarray}\nonumber\label{e5.16}
   \sup_{\Omega\in\partial{\mathcal{K}}_{p,f}(\delta,{\mathcal{S}}_T)}
   {\mathcal{F}}_{p}(\Omega)&\leq&\sup_{\Omega\in\partial{\mathcal{K}}
   (\delta,{\mathcal{S}}_T)}{\mathcal{F}}_{p}(\Omega)\\ \nonumber
   &=&\sup_{\Omega\in\partial{\mathcal{K}}(\delta,{\mathcal{S}}_T)}
   {\mathcal{F}}_{-\infty}(\Omega)+\varepsilon(p)\\ \nonumber
     &\leq&{\mathcal{F}}_{-\infty}(T)-\sigma+\varepsilon(p)\\ \nonumber
     &=&\sup_{\Omega\in{\mathcal{K}}_{p,f}(\delta,{\mathcal{S}}_T)}V(\Omega)-\sigma
     +\varepsilon(p)
   \end{eqnarray}
hold for negative large $p$ and some infinitesimal $\varepsilon(p)$, where in the second inequality we have also used \eqref{e5.12} and the convergence of $\partial{\mathcal{K}}_{p,f}(\delta,{\mathcal{S}}_T)$ to $\partial{\mathcal{K}}_{min}(\delta,{\mathcal{S}}_T)$ under Hausdorff distance. Using again the relation
   \begin{equation}
     {\mathcal{F}}_{p}(\Omega)={\mathcal{F}}_{-\infty}(\Omega)+\varepsilon(p)=V(\Omega)+\varepsilon(p), \ \ \forall\Omega\in{\mathcal{K}}_{p,f}(\delta,{\mathcal{S}}_T)
   \end{equation}
for some infinitesimal $\varepsilon(p)$ and negative large $p$, we conclude from \eqref{e5.16} that
    \begin{equation}\label{e5.17}
      \sup_{\Omega\in\partial{\mathcal{K}}_{p,f}(\delta,{\mathcal{S}}_T)}V(\Omega)
      \le\sup_{\Omega\in{\mathcal{K}}_{p,f}(\delta,{\mathcal{S}}_T)}
      V(\Omega)-\sigma+\varepsilon(p)
    \end{equation}
holds for small $\delta>0$ and negative large $p$. We thus arrive at the existence result of local maximizer of the functional ${\mathcal{F}}_{p}(\cdot)$ for the negative large $p$.
\end{proof}

\begin{lemm}\label{l5.3}
 Let $T\subset{\mathbb{R}}^{n+1}$ be a regular polytope tangential to the unit sphere ${\mathbb{S}}^n$.
 If $\delta=\delta_{T}>0$ is small and $p$ is negative large,
 the variational problem \eqref{e5.11} admits a maximizer $\Omega^{(p)}\not\in \p{\mathcal{N}}_\delta(T)$.
 Moreover, the maximizer $\Omega^{(p)}$ satisfies the equation \eqref{e1.2}, and converges to the given regular polytope $T$ as $p\rightarrow -\infty$.
\end{lemm}

\begin{proof}
The existence of maximizer $\Omega^{(p)}$ follows from Lemma \ref{l5.2}.  Owing to  the property that $\Omega^{(p)}$ stays away from the boundary of ${\mathcal{N}}_\delta(T)$, we can prove that $\Omega^{(p)}$ satisfies the Euler-Lagrange equation \eqref{e1.2} as in the proof of global variational problem. Noting that $C(n,p,f,{\mathcal{S}_T})$ converges to $C(n,-\infty,f,{\mathcal{S}_T})$ as $p\to-\infty$, the maximizer $\Omega^{(p)}$ converges to a maximizer $\Omega^{(-\infty)}$ of the limiting functional. However, by Lemma \ref{l5.1}, $T$ is the unique maximizer of the limiting functional on the family ${\mathcal{K}}_{min}(\delta,{\mathcal{S}}_T)$. This, together with the observation that $\mathcal{F}_{-\infty}(\Omega)=V(\Omega)$ for all $\Omega\in{\mathcal{K}}_{min}(\delta,{\mathcal{S}}_T)$, implies that $\Omega^{(-\infty)}=T$. This completes the proof of the lemma.
\end{proof}

\begin{proof}[Proof of Theorem \ref{t1.2}.] Theorem \ref{t1.2} is a direct consequence of Lemma \ref{l5.3}.
\end{proof}

From our proof, a regular polytope is a local maximizer of the functional $\mathcal F_p$.
We would like to point out that it may not be a global maximizer  of $\mathcal F_p$.

\vspace{10pt}

\section{The $L_p$ dual Minkowski problem}

In this section, we turn to study the $L_p$ dual Minkowski problem \eqref{e1.1}
and prove Theorem \ref{t1.3}. Parallel to Section 3, we consider a variational problem
   \begin{equation}\label{e6.1}
      \sup_{\Omega\in{\mathcal{K}}_{p}({\mathcal{S}})}V_q(\Omega) \ \ \text{with}\ \ V_q(\Omega)\equiv\int_{{\mathbb{S}}^{n}}r^q(y)dy, \ \ q\not=0
   \end{equation}
on the group ${\mathcal{S}}$ invariant family
   $$
    {\mathcal{K}}_{p}({\mathcal{S}}):=\Bigg\{\Omega\in{\mathcal{K}}_0({\mathcal{S}})\Big|\ \int_{{\mathbb{S}}^n}fh_\Omega^p=\int_{{\mathbb{S}}^n}f\Bigg\}, \ \ p\not=0.
   $$
A similar argument of Lemma \ref{p3.1} gives the following solvability result.

\begin{prop}\label{p6.1}
Let ${\mathcal{S}}$ be a discrete subgroup of $O(n+1)$ satisfying the spanning property.
Then for  $q\not=0$ and $p<-n$, the variational problem \eqref{e6.1} admits a smooth positive maximizer $h^{(p,q)}$ satisfying the Euler-Lagrange equation \eqref{e1.1} up to a constant.
\end{prop}

Recall the functional
\begin{equation*}
 \F_{p,q}(\Omega) = V_q(\Omega)\left(\int_{{\mathbb{S}}^{n}}fh_\Omega^p{\Big/}\int_{{\mathbb{S}}^{n}}f\right)^{-q/p}
 \end{equation*}
 defined on $\mathcal{S}$-invariant convex bodies $\Om\subset \mathcal K_0(\mathcal{S})$.
 Similarly to Section 4, the maximizer $\Omega_h^{(p,q)}$ in Proposition \ref{p6.1}
 also serves as the extremal body of the following inequality.

\begin{prop}\label{p6.2}
 Let ${\mathcal{S}}$ be a discrete subgroup of $O(n+1)$ satisfying the spanning property.
 Then for $q\not=0$ and $p<-n$, we have
   \begin{equation}\label{e6.2}
      \mathcal{F}_{p,q}(\Omega)
       \leq C(n,p,q,f,{\mathcal{S}}), \ \ \forall\ \Omega\in{\mathcal{K}}_0({\mathcal{S}})
   \end{equation}
with
 $$C(n,p,q,f,{\mathcal{S}}):=\sup_{\Omega\in{\mathcal{K}}_{p}({\mathcal{S}})} \mathcal{F}_q(\Omega).$$
Moreover, the inequality \eqref{e6.2} becomes equality
at some smooth group invariant solution $h^{(p,q)}$ of \eqref{e1.1}, corresponding to the group invariant convex body $\Omega^{(p,q)}\in{\mathcal{K}}_0({\mathcal{S}})$.
\end{prop}

\begin{rema}\label{r6.1}
 Let $\Omega^*=\{x\in\mathbb{R}^{n+1}:x\cdot y\le1\ \text{for}\ \forall\ y\in \Omega\}$ be the polar body of $\Omega$. Then there exists a duality
   \begin{equation}
     \mathcal{F}_{-q,-p}(\Omega)=\mathcal{F}^*_{p,q}(\Omega^*)
     := \left(\int_{{\mathbb{S}}^n}h_{\Omega^*}^p\right)
     \left(\int_{{\mathbb{S}}^n}fr_{\Omega^*}^q{\Big/}
     \int_{{\mathbb{S}}^n}f\right)^{-p/q}.
   \end{equation}
 Thus, the critical point $\Omega$ of the functional $\mathcal{F}_{-q,-p}(\cdot)$ corresponds to the critical point $\Omega^*$ of the functional $\mathcal{F}^*_{p,q}(\cdot)$. Consequently, an analogous solvability result to Proposition \ref{p6.1} also holds for $p\not=0$ and $q>n$.
\end{rema}

By the group invariance of the solution $\Omega^{(p,q)}$,
one has the uniform {\it a-priori} bound for $\Omega^{(p,q)}$.
Passing to the limit $p\to-\infty$,
we obtain the limit $\Omega^{(-\infty,q)}$ which is the extremal body of the functional $ \mathcal{F}_{-\infty,q}$.

\begin{lemm}\label{l6.1}
   Let ${\mathcal{S}}$ be a discrete subgroup of $O(n+1)$ satisfying the spanning property. The best constant $C(n,p,q,f,{\mathcal{S}})$ of the functional $\mathcal{F}_{p,q}(\cdot)$ converges to the best constant $C(n,-\infty,q, {\mathcal{S}})$ of the limiting functional
     $$
      \mathcal{F}_{-\infty,q}(\Omega):= V_q(\Omega){\Big/}(\min_{{\mathbb{S}}^{n}}h_\Omega)^{q}, \ \ \Omega\in{\mathcal{K}}_0({\mathcal{S}}).
     $$
Moreover, $\Omega^{(-\infty,q)}$ is a maximizer of the functional $\mathcal{F}_{-\infty,q}(\cdot)$.
\end{lemm}

The proof of Lemma \ref{l6.1} is similar to that of Lemma \ref{l4.1} and is omitted here.
Next we show

\begin{lemm}\label{l6.2}
 Let ${\mathcal{S}}$ be a discrete subgroup of $O(n+1)$ satisfying the spanning property.
 \begin{itemize}
\item [(1)] The maximizer $\Omega^{(-\infty,q)}$ is a group invariant polytope if $q>0$.
\item [(2)] The maximizer $\Omega^{(-\infty,q)}$ is a ball $B_R\subset{\mathbb{R}}^{n+1}$ for some $R>0$, if $q<0$.
\end{itemize}
\end{lemm}

\begin{proof}
Part (1) can be proven as that in Theorem \ref{t4.2}.
For part (2), let $\Omega^*$  be the polar body of $\Omega$. Then,
$$
\mathcal{F}_{-\infty,q}(\Omega)=\mathcal{F}^*_{-q}(\Omega^*)
        :=\left(\int_{{\mathbb{S}}^n}h_{\Omega^*}^{-q}\right)
{\Big/}(\max_{{\mathbb{S}}^n}r_{\Omega^*})^{-q}, \ \ \Omega\in{\mathcal{K}}_0({\mathcal{S}}).
$$
Noting that the functional $\mathcal{F}^*_{-q}(\cdot)$ is invariant under scaling, we may assume that $\max_{{\mathbb{S}}^n}r_{\Omega^*}=1$. In this case, we have
$$
\mathcal{F}^*_{-q}(\Omega^*)=\int_{{\mathbb{S}}^n}h_{\Omega^*}^{-q}\le (n+1)\omega_{n+1},
$$
with equality if and only if $\Omega^*={B_1}$. This implies that the extremal body $\Omega^*$ maximizing the functional $\mathcal{F}^*_{-q}(\cdot)$ must be a ball. Hence, we conclude that $\Omega^{(-\infty,q)}$ is also a ball by duality. The proof is completed.
\end{proof}

\begin{lemm}\label{l6.3}
 Let ${\mathcal{S}}$ be a discrete subgroup of $O(n+1)$ satisfying the spanning property. For each $p\neq0$, the limiting shape $\Omega^{(p,+\infty)}$ of $\Omega^{(p,q)}$ as $q\to+\infty$  is a group invariant polytope.
\end{lemm}
\begin{proof}
When $p<0$ and $q>n$, the maximizer of $\mathcal{F}_{p,q}(\cdot)$ is also the maximizer of the functional
 $$
  \mathcal{F}^*_{p,q}(\Omega):=\left(\int_{{\mathbb{S}}^{n}}r_\Omega^{q}\right)^{-p/q}\left(\int_{{\mathbb{S}}^{n}}fh_\Omega^p{\Big/}
  \int_{{\mathbb{S}}^{n}}f\right).
 $$
As $q\to+\infty$, we conclude that $\Omega^{(p,+\infty)}$ is the maximizer of the limiting functional
  $$
  \mathcal{F}^*_{p,+\infty}(\Omega):=
  \left(\int_{{\mathbb{S}}^{n}}fh_\Omega^p{\Big/}\int_{{\mathbb{S}}^{n}}f\right)
  {\Big/}(\max_{{\mathbb{S}}^n}r_\Omega)^p.
  $$
Noting that the function $\mathcal{F}^*_{p,+\infty}(\cdot)$ is invariant under scaling, we may assume that $\max r_\Omega=\max h_\Omega=h(x_0)=1$ for some $x_0\in \mathbb{S}^n$.
Let $P_{x_0}=\text{conv}\{\phi(x_0):\phi\in{\mathcal{S}}\}\subset {B_1}$ be a group invariant polytope.
It is clear that $P_{x_0}$ is contained inside of $\Omega$ by convexity. Thus,
   $$
   \mathcal{F}^*_{p,+\infty}(\Omega)<\mathcal{F}^*_{p,+\infty}(P_{x_0})
   $$
as long as $\Omega$ is not a polytope. Therefore, the limiting shape $\Omega^{(p,+\infty)}$ can only be a polytope for $p<0$.

When $p>0$ and $q>n$, the maximizer of $\mathcal{F}_{p,q}(\cdot)$ is also the minimizer of the functional $\mathcal{F}^*_{p,q}(\cdot)$. Passing to the limit as $q\to+\infty$, we conclude that $\Omega^{(p,+\infty)}$ is the minimizer of the limiting functional $\mathcal{F}^*_{p,+\infty}(\cdot)$. Without loss of generality, we may also assume that $\max r_\Omega=r_\Omega(a)=1$ for some $a\in \mathbb{S}^n$. Then, it is clear that the minimizer of the functional $\mathcal{F}^*_{p,+\infty}(\cdot)$ is attained at $P_a$ for $p>0$.
\end{proof}

\begin{proof}[Proof of Theorem \ref{t1.3}]
 Let $\Omega^{(p,q)}$ be the solution obtained in Proposition \ref{p6.1} with the support function $h^{(p,q)}$. Along the same lines as in Corollary \ref{c3.1}, we have for $q\neq 0$
   \begin{equation}\label{e6.3}
     \lim_{p\to-\infty}\min_{{\mathbb{S}}^n}h^{(p,q)}=1.
   \end{equation}
This, together with Lemma \ref{l6.2}, yields the cases (1) and (2) of Theorem \ref{t1.3}.

 Let $\Omega^{(p,q)}$ be the solution obtained in Proposition \ref{p6.1} with the support function $h^{(p,q)}$ for $p\not=0$ and $q>n$ as stated in Remark \ref{r6.1}. Along the same lines as in Corollary \ref{c3.1}, we have for $p\neq 0$
   \begin{equation*}
     \lim_{q\to  +\infty}\min_{{\mathbb{S}}^n}h^{(p,q)}=1.
   \end{equation*}
This, together with Lemma \ref{l6.3}, yields the case (3) of Theorem \ref{t1.3}.
\end{proof}

\vspace{10pt}

\section{The proof of Theorem \ref{t1.4}}

This section aims to prove Theorem \ref{t1.4}. Firstly, let us consider $q>0$.
In this case, we have the following result, analogous to Lemma \ref{l5.1}.

\begin{lemm}\label{l7.1}
For $q>0$, a regular polytope $T$ is a strict local maximal of the functional $V_q$
in the family ${\mathcal{K}}_{min}(\delta,{\mathcal{S}}_T)$ for some $\delta>0$.
Namely,  there exists a positive constant $\delta=\delta_{q,T}$
such that for any $\Omega'\in{\mathcal{K}}_{min}(\delta,{\mathcal{S}}_T)\setminus\{T\}$, there holds
    \begin{equation}\label{e7.1}
      V_q(\Omega')<V_q(T).
    \end{equation}
\end{lemm}

\begin{proof}
Similar to the proof of Claim $\sharp1$ in  Lemma \ref{l5.1},
we can check that for some $\delta>0$,
$$
V_q(T)=\sup_{\Omega\in{\mathcal{K}}_{min}(\delta,{\mathcal{S}}_T)}V_q(\Omega)
$$
if and only if
  \begin{equation}\label{e7.2}
  \displaystyle  \overline{V}_q(\overline{w}_0)=\int_{\Omega_n}(1+|\overline{z}|^2)^{\frac{q-n-1}{2}}
  \bigg(1-\bigg|\frac{\sqrt{1+|\overline{w}_0|^2}}{1+\langle \overline{z},\overline{w}_0\rangle}\bigg|^{q}\bigg)d\overline{z}>0,
  \end{equation}
 for all $\overline{w}_0\in\Omega_n\setminus\{0\}$ closing to $0$.

Let $w_0=(\overline{w}_0,1)$ approach $O_n=e_{n+1}$ but $w_0\ne O_n$. Along the direction segment $I=\{O_n+ra\in\Omega_n| r\geq0, a\in{\mathbb{R}}^n_+\}$, the volume difference function can be written as
    $$
     V_q(r,a):=\overline{V}_q(\overline{w}_0)
     =\int_{\Omega_n}(1+|\overline{z}|^2)^{\frac{q-n-1}{2}}
       \bigg(1-\bigg|\frac{\sqrt{1+r^2|a|^2}}{1+r\langle \overline{z},a\rangle}\bigg|^{q}\bigg)d\overline{z}.
    $$
Then by the fact that $\Omega_{n}\subset{\mathbb{R}}^{n}_+$, we have
    \begin{equation}\nonumber
     \frac{d}{dr}\bigg|_{r=0}V_q(r,a)=q\int_{\Omega_n}(1+|\overline{z}|^2)^{\frac{q-n-1}{2}}\langle\overline{z},a\rangle d\overline{z}>0,
        \ \ \forall a\in{\mathbb{R}}^{n}_+.
    \end{equation}
This, together with \eqref{e7.2}, gives the desired inequality \eqref{e7.1}.
\end{proof}

To complete the proof of Theorem \ref{t1.4} for the case $q>0$, we shall consider the local variational scheme
    \begin{equation}\label{e7.3}
       \sup_{\Omega\in{\mathcal{K}}_{p,f}(\delta,{\mathcal{S}}_T)}V_q(\Omega).  \end{equation}
Then we will show the existence of maximizer $\Omega^{(p,q)}\in{\mathcal{K}}_{p,f}(\delta,{\mathcal{S}}_T)$ of \eqref{e7.3} for negative large $p$, whose support function $h^{(p,q)}$ satisfies the Euler-Lagrange equation \eqref{e1.1} up to a constant.

\begin{lemm}\label{l7.2}
 Let $q>0$ and $T\subset{\mathbb{R}}^{n+1}$ be a regular polytope tangential to the unit sphere ${\mathbb{S}}^n$.
 If $\delta=\delta_{q,T}>0$ is small and $p$ is negative large, the variational scheme \eqref{e7.3} admits
 a maximizer $\Omega^{(p,q)}\not\in \p {\mathcal{N}}_\delta(T)$.
 Moreover, the maximizer $\Omega^{(p,q)}$ satisfies the equation \eqref{e1.1}, and converges to the given regular polytope $T$ as $p\rightarrow -\infty$.
\end{lemm}
\begin{proof}
The existence of maximizer $\Omega^{(p,q)}$ to the variational scheme \eqref{e7.3} follows from {\it a-priori} estimates for group invariant polytopes.

Next, we show that the property
that $\Omega^{(p,q)}\not\in \p {\mathcal{N}}_\delta(T)$.
In fact, by Lemma \ref{l7.1}, we know that $T$ is a local strictly maximizer of the functional ${\mathcal{F}}_{-\infty,q}(\cdot)$
on the family ${\mathcal{K}}_{min}(\delta,{\mathcal{S}}_T)$ for some $\delta>0$.
Hence, there exists a positive constant $\delta=\delta_{q,T}$ such that
   \begin{equation}\label{e7.4}
      {\mathcal{F}}_{-\infty,q}(\Omega)<{\mathcal{F}}_{-\infty,q}(T), \ \ \forall\Omega\in{\mathcal{K}}_{min}(\delta,{\mathcal{S}}_T)\setminus T.
   \end{equation}
Thus we have that for some $\sigma>0$,
    \begin{equation}\label{e7.5}
      \sup_{\Omega\in\partial{\mathcal{K}}_{min}(\delta,{\mathcal{S}}_T)}
      {\mathcal{F}}_{-\infty,q}(\Omega)\leq{\mathcal{F}}_{-\infty,q}(T)-\sigma,
    \end{equation}
and
  \begin{equation}\label{e7.6}
      \lim_{p\rightarrow -\infty}
      {\mathcal{F}}_{p,q}(\Omega)={\mathcal{F}}_{-\infty,q}(\Omega).
    \end{equation}
Similar to the proof of Lemma \ref{l5.2}, we have
\begin{eqnarray*}
   \sup_{\Omega\in\partial{\mathcal{K}}_{p,f}(\delta,{\mathcal{S}}_T)}
   {\mathcal{F}}_{p,q}(\Omega)
   \leq\sup_{\Omega\in{\mathcal{K}}_{p,f}(\delta,{\mathcal{S}}_T)}V_q(\Omega)-\sigma
     +\varepsilon(p)
   \end{eqnarray*}
holds for negative large $p$ and some infinitesimal $\varepsilon(p)$, and hence
    \begin{equation}\label{e7.7}
      \sup_{\Omega\in\partial{\mathcal{K}}_{p,f}(\delta,{\mathcal{S}}_T)}V_q(\Omega)
      \leq\sup_{\Omega\in{\mathcal{K}}_{p,f}(\delta,{\mathcal{S}}_T)}
      V_q(\Omega)-\sigma+\varepsilon(p)
    \end{equation}
holds for small $\delta>0$ and negative large $p$.
This implies that the property that $\Omega^{(p,q)}$
stays away from the boundary of ${\mathcal{N}}_\delta(T)$.
Then we can prove that $\Omega^{(p,q)}$ satisfies the equation \eqref{e1.1}
as in the proof of the variational problem \eqref{e6.1}.
Since
$$
\lim_{p\to-\infty}\sup_{\Omega\in{\mathcal{K}}_{p,f}(\delta,{\mathcal{S}}_T)}V_q(\Omega)
=\sup_{\Omega\in{\mathcal{K}}_{min}(\delta,{\mathcal{S}}_T)}V_q(\Omega),
$$
by the compactness of ${\mathcal{K}}(\delta,{\mathcal{S}}_T)$.
Thus, the maximizer $\Omega^{(p,q)}$ converges to a maximizer $\Omega^{(-\infty,q)}$ of the limiting functional.
However, by Lemma \ref{l7.1}, $T$ is the unique maximizer of the limiting functional on
the family ${\mathcal{K}}_{min}(\delta,{\mathcal{S}}_T)$.
We reach the conclusion that $\Omega^{(-\infty,q)}=T$ by the observation
that ${\mathcal{F}}_{-\infty,q}(\Omega)=V_q(\Omega)$ for $\forall\ \Omega\in \mathcal{K}_{min}(\delta,\mathcal{S}_T)$.
This completes the proof.
\end{proof}

\begin{proof}[Proof of (1) in Theorem \ref{t1.4}.] Part (1) of Theorem \ref{t1.4} is a direct consequence of Lemma \ref{l7.2}.
\end{proof}

To prove the second part of Theorem \ref{t1.4}, we denote
   $$
      V_{p,f}(\Omega)=\int_{{\mathbb{S}}^n}fh_\Omega^p, \ \ \ p\not=0,\ f>0,
   $$
and
    $${\begin{split}
     {\mathcal{K}}_{max}({\mathcal{S}}_T) & =\big\{\Omega\in{\mathcal{K}}_0({\mathcal{S}}_T)\big| \max_{{\mathbb{S}}^n}h_\Omega=1\big\},\\
    {\mathcal{K}}_{max}(\delta,{\mathcal{S}}_T)&={\mathcal{K}}_{max}({\mathcal{S}}_T)\cap{\mathcal{N}}_\delta(T).
    \end{split}}
    $$

\begin{lemm}\label{l7.3}
 For any regular polytope $T\subset{\mathbb{R}}^{n+1}$ with outer radius one and positive group invariant function $f$ on ${\mathbb{S}}^n$, there exists a small constant $\delta=\delta_{p,f,T}>0$ such that
     \begin{equation}\label{e7.8}
      \begin{cases}
       V_{p,f}(T)<V_{p,f}(\Omega), & p>0, \ \ \forall\Omega\in{\mathcal{K}}_{max}(\delta,{\mathcal{S}}_T)\setminus\{T\},\\
        V_{p,f}(T)>V_{p,f}(\Omega), & p<0, \ \ \forall\Omega\in{\mathcal{K}}_{max}(\delta,{\mathcal{S}}_T)\setminus\{T\}.
      \end{cases}
     \end{equation}
\end{lemm}

\begin{proof}
Firstly, let $p>0$ and assume $f$ is positive group invariant function. For any $\Omega\in{\mathcal{K}}_{max}(\delta,{\mathcal{S}}_T)$, there exists at least one $a\in{\mathbb{S}}^n$ such that
   \begin{equation*}
      h_\Omega(a)=\max_{{\mathbb{S}}^n}h_\Omega=1.
   \end{equation*}
Set $\Omega_a=\text{conv}(\cup_{\phi\in{\mathcal{S}}_T}\phi(a))$, it is clearly that
    \begin{equation}\label{e7.9}
    V_{p,f}(\Omega)\geq V_{p,f}(\Omega_a)
    \end{equation}
since $\Omega_a\subset\Omega$. We remain to show that
    \begin{equation}\label{e7.10}
      V_{p,f}(\Omega_a)>V_{p,f}(T)
    \end{equation}
if $\Omega_a$ approaches $T$ without being equal to $T$. Noting that $\Omega_a, T\in{\mathcal{K}}_{max}(\delta,{\mathcal{S}}_T)$, if one takes their dual bodies $\Omega_b$ (for some $b\in \mathbb{S}^n$), $T^*$ respectively, we have $\Omega_b, T^*\in{\mathcal{K}}_{min}(\delta,{\mathcal{S}}_{T^*})$ are closing from each other under the Hausdorff distance.
So, to prove \eqref{e7.10} is equivalent to proving
    \begin{equation}\label{e7.11}
     V^*_{-p,f}(\Omega_b)>V^*_{-p,f}(T^*),
    \end{equation}
where
$$
V^*_{q,f}(\Omega):=\int_{{\mathbb{S}}^n}f(y)r_\Omega^q(y)dy.
$$

The proof of \eqref{e7.11} is similar to the proof of \eqref{e5.8}. Actually, after subdivision of $T^*$ into congruent pieces similar to one polytope $\Omega_n\subset{\mathbb{R}}^{n}_+$, the comparison of $V^*_{-p,f}(\Omega_b)$ with $V^*_{-p,f}(T^*)$ can be reduced to the verification of the negativity of the function
\begin{equation*}
  \displaystyle  \overline{V}_{-p,f}(\overline{w}_0)
  =\int_{\Omega_n}f(1+|\overline{z}|^2)^{\frac{-p-n-1}{2}}
  \bigg(1-\bigg|\frac{\sqrt{1+|\overline{w}_0|^2}}{1+\langle \overline{z},\overline{w}_0\rangle}\bigg|^{-p}\bigg)d\overline{z}
\end{equation*}
for $\overline{w}_0\not=0$ approaching $0$. Along the direction segment $I=\{O_n+ra\in\Omega_n| r\geq0, a\in{\mathbb{R}}^n_+\}$, we have
   $$
   V_{-p,f}(r,a): = \overline{V}_{-p,f}(\overline{w}_0)=\int_{\Omega_n}f(1+|\overline{z}|^2)^{\frac{-p-n-1}{2}}\bigg(1-\bigg|\frac{\sqrt{1+r^2|a|^2}}{1+r\langle \overline{z},a\rangle}\bigg|^{-p}\bigg)d\overline{z}.
   $$
Therefore, it follows from
   \begin{equation}
    \frac{d}{dr}\bigg|_{r=0}V_{-p,f}(r,a)=-p\int_{\Omega_n}f(1+|\overline{z}|^2)^{\frac{-p-n-1}{2}}\langle\overline{z},a\rangle d\overline{z}<0
   \end{equation}
for $a\not=0$ that \eqref{e7.11} holds true.

After modifying the proof of \eqref{e7.11} slightly, we can also show that:
If $p<0$, for $\Omega_b\in{\mathcal{K}}_{min}(\delta,{\mathcal{S}}_{T^*}), b\in{\mathbb{S}}^n$ approaches $T^*$ without being equal to $T^*$, there holds
    \begin{equation*}
     V^*_{-p,f}(\Omega_b)<V^*_{-p,f}(T^*),
    \end{equation*}
which is equivalent to
\begin{equation*}
     V_{-p,f}(\Omega_a)<V_{-p,f}(T).
    \end{equation*}
This, together with the fact that
    \begin{equation}
     V_{p,f}(\Omega)\leq V_{p,f}(\Omega_a), \ \ \forall\ \Omega\in{\mathcal{K}}_{max}(\delta,{\mathcal{S}}_T),
    \end{equation}
gives the desired result for case $p<0$.
\end{proof}

With the help of Lemma \ref{l7.3}, we know that the regular polytope $T$ is of local strictly extremal of $V_{p,f}$ for $p\not=0$ and $f>0$. Considering the same local variational scheme \eqref{e7.3} and modifying the argument in Section 5 slightly, we reach the following result.

\begin{lemm}\label{l7.4}
 Let $p\not=0$ and $T\subset{\mathbb{R}}^{n+1}$ be a regular polytope tangential with outer radius one,
 if $\delta=\delta_{p,T}>0$ is small and $q$ is positive large,
 the variational scheme \eqref{e7.3} admits a maximizer $\Omega^{(p,q)}\not\in {\p\mathcal{N}}_\delta(T)$.
 Moreover, the maximizer $\Omega^{(p,q)}$ satisfies the equation \eqref{e1.1},
 and converges to a regular polytope similar to $T$ as $q\rightarrow +\infty$.
\end{lemm}

\begin{proof}[Proof of (2) in Theorem \ref{t1.4}.] Part (2) of Theorem \ref{t1.4} is a direct consequence of Lemma \ref{l7.4}.
\end{proof}

\vspace{10pt}

\noindent {\bf Acknowledgement.} The authors would like to thank Xudong Wang for many
helpful comments.

\end{document}